\documentclass{article}  
\usepackage{pinlabel}
\usepackage[all]{xy}
\usepackage{amssymb,amsthm,amsmath}
\usepackage{euscript}
\usepackage{graphicx}
\usepackage{color}
\usepackage{hyperref}

\title{Exotic iterated Dehn twists}
\date{}
\author{Paul Seidel} 
\newcommand{\bR}{\mathbb{R}}
\newcommand{\bC}{\mathbb{C}}
\newcommand{\bZ}{\mathbb{Z}}

\newcommand{\htp}{\simeq}
\newcommand{\iso}{\cong}
\newcommand{\smooth}{C^\infty}
\newcommand{\myhalf}{\frac{1}{2}}
\newcommand{\Symp}{\mathit{Symp}}
\newcommand{\Diff}{\mathit{Diff}}
\newcommand{\scrM}{\mathcal{M}}
\newcommand{\scrK}{\mathcal{K}}

\numberwithin{equation}{section}

\newtheorem{theorem}{Theorem}[section]
\newtheorem{definition}[theorem]{Definition}
\newtheorem{remark}[theorem]{Remark}
\newtheorem{lemma}[theorem]{Lemma}

\newtheorem{assumption}[theorem]{Assumption}

\begin{document}
\maketitle
\begin{abstract}
Consider cotangent bundles of exotic spheres, with their canonical symplectic structure. They admit automorphisms which preserve the part at infinity of one fibre, and which are analogous to the square of a Dehn twist. Pursuing that analogy, we show that they have infinite order up to isotopy (inside the group of all automorphisms with the same behaviour).
\end{abstract}

\section{Introduction}

Let $M = T^*L$ be the cotangent bundle of a closed connected oriented $n$-manifold $L$, with its standard symplectic structure $\omega_M = d\theta_M$. This is a Liouville manifold, modelled at infinity on the positive part of the symplectization of the contact sphere bundle $\partial_\infty M = S^*L$. Let $\Symp(M)$ be the group of symplectic automorphisms $\phi$ which are of contact type at infinity, or equivalently such that $\phi^*\theta_M - \theta_M$ is closed and has compact support. Fix some point $q \in L$. Write $F = T^*_qL \subset M$ for the cotangent fibre at that point, and $K = S^*_qL = \partial_\infty F \subset \partial_\infty M$ for the corresponding sphere fibre. Let
\begin{equation} \label{eq:symp-q}
\Symp^K(M) \subset \Symp(M)
\end{equation}
be the subgroup of those $\phi$ such that the submanifolds $F$ and $\phi(F)$ agree outside a compact subset, including orientations; equivalently, the associated contact diffeomorphism $\partial_\infty \phi$ of $\partial_\infty M$ maps $K$ to itself orientation-preservingly.

\begin{theorem} \label{th:main}
Let $L$ be an exotic $n$-sphere, $n \neq 4$. Then $\pi_0(\Symp^K(M))$ contains an element of infinite order.
\end{theorem}

By an exotic $n$-sphere, we mean a homotopy sphere not diffeomorphic to $S^n$. There are no such spheres for $n = 1,2,3,5,6$, so Theorem \ref{th:main} concerns $n \geq 7$. We will give two parallel but logically independent constructions of elements of $\Symp^K(M)$: one using Riemannian geometry and the geodesic flow, and the other via Lefschetz fibrations and mono\-dromy. Each of them imitates a certain interpretation of the square of a Dehn twist $\tau^2$ on the cotangent bundle of the standard sphere. We will not actually prove that the two approaches have the same outcome, but we will provide a Floer-theoretic proof of Theorem \ref{th:main} which, with minor adjustments, applies to either of them. The main constructions are contained in Sections \ref{sec:first-def} and \ref{sec:second-def}, respectively; the rest of the paper explains the relevant background material (which is not new, but included for convenience).

{\em Acknowledgments.} Matthias Kreck helped me to clarify the discussion of homotopy spheres. Ailsa Keating suggested the heuristic argument from Section \ref{sec:heuristic}. Xin Jin and the anonymous referee contributed helpful comments on the first version of this manuscript (that version eventually turned out to have a serious flaw; see Remark \ref{th:eating-crow}). I am grateful to all of them. This work was partially supported by a Simons Investigator grant from the Simons Foundation, as well as by NSF grant DMS-1005288.

\section{Homotopy spheres\label{sec:spheres}}

A homotopy $n$-sphere is a closed oriented manifold $S$ homotopy equivalent to $S^n$. From now on, when talking about such manifolds, we always assume that $n \geq 6$. Let $\Theta_n$ be the set of diffeomorphism classes of homotopy $n$-spheres. This is a finite abelian group under connected sum, with the inverse map being orientation-reversal $S \mapsto -S$ \cite{kervaire-milnor63}. Given $f \in \Diff^+(S^{n-1})$, one can form the homotopy $n$-sphere
\begin{equation} \label{eq:glue-2-discs}
S_f = B_- \cup_f B_+.
\end{equation}
Here, $B_-$ and $B_+$ are two copies of the closed unit ball in $\bR^n$, and one identifies $q \in \partial B_-$ with $f(q) \in \partial B_+$ (to fix the differentiable structure on $S_f$ explicitly, let's say that we use radial collar neighbourhoods of $\partial B_\pm$). Equip $S_f$ with the orientation coming from $B_-$ (the opposite of the orientation coming from $B_+$). This yields a group homomorphism
\begin{equation} \label{eq:diff-to-theta}
\pi_0(\Diff^+(S^{n-1})) \longrightarrow \Theta_n.
\end{equation}
The $h$-cobordism theorem shows that \eqref{eq:diff-to-theta} is onto, and Cerf's theorem \cite{cerf70} that it is injective. Here are two useful consequences of the representability of homotopy spheres in the form \eqref{eq:glue-2-discs}:

\begin{lemma} \label{th:metric}
There is a Riemannian metric on $S_f$ with the following property. For $q \in S^{n-1} = \partial B_\pm$, consider the path which consists of a straight line segment (with its standard parametrization) from $q_- = 0 \in B_-$ to $q \in \partial B_-$, and another such segment from $f(q) \in \partial B_+$ to $q_+ = 0 \in B_+$. Then all those paths are geodesics (of the same speed).
\end{lemma}

This is due to Weinstein, see \cite{warner67} and \cite[Appendix C]{besse} for expositions. For greater compatibility with the standard sphere case, we will assume that the unit speed geodesics are those which take time $\pi$ to go from $q_-$ to $q_+$. This is unproblematic, since it can be achieved simply by rescaling the metric.

%Conversely, given an exotic sphere, if one can explicitly write down a metric whose geodesics show this kind of ``wiedersehen'' behaviour 
%
%Conversely, if one can explicitly construct metric on a given homotopy sphere with such a property, 
%
%Take geodesic flow (thought of as living on the total space of the tangent bundle) for such a metric yields a diffeomorphism
%
%If one thinks of the geodesic flow as living on $TM$, then the flow
%
%a metric as in Lemma \ref{th:metric}
%
%the derivative of the geodesic flow at $x_-$ 
%
%
%
%explicit formulae have been given for the diffeomorphisms giving rise to some exotic spheres \cite{?}.

\begin{lemma} \label{th:conjugate-with-reflection}
Take $f \in \Diff^+(S^{n-1})$, and an orientation-reversing $r \in O(n)$. Then $rfr^{-1}$  is isotopic to $f^{-1}$.
\end{lemma}

\begin{proof}
Applying $r$ to both hemispheres in \eqref{eq:glue-2-discs} yields a diffeomorphism $-S_f \rightarrow S_{rfr^{-1}}$. On the other hand $-S_f \iso S_{f^{-1}}$, by exchanging the hemispheres. Because of the injectivity of \eqref{eq:diff-to-theta}, this implies the desired result.
\end{proof}

\section{Cotangent bundles}

Consider a cotangent bundle $M = T^*L$. Here $L$ is closed connected oriented, as in the Introduction, and we again fix a point $q$ together with the associated submanifolds $F \subset M$, $K \subset \partial_\infty M$. Consider the map
\begin{equation} \label{eq:induced-k}
\begin{aligned}
& \Symp(M) \longrightarrow \scrK(\partial_\infty M), \\
& \phi \longmapsto \partial_\infty\phi(K),
\end{aligned}
\end{equation}
where $\scrK(\partial_\infty M)$ is the space of oriented Legendrian spheres in $\partial_\infty M$. Clearly, $\Symp^K(M)$ is the fibre of \eqref{eq:induced-k} over $K$. Since that map has the lifting property with respect to smooth paths, one gets a boundary homomorphism
\begin{equation} \label{eq:boundary-map}
\pi_1(\scrK(\partial_\infty M), K) \longrightarrow \pi_0(\Symp^K(M)).
\end{equation}
Any element of $\Symp^K(M)$ induces an orientation-preserving diffeomorphism of $K$. This yields a sequence (exact at the middle terms)
\begin{equation} \label{eq:map-to-theta-n}
\pi_1(\Diff^+(S^{n-1})) \rightarrow \pi_0(\Symp^{K,\mathit{fix}}(M)) \rightarrow \pi_0(\Symp^K(M)) \rightarrow \pi_0(\Diff^+(S^{n-1})),
\end{equation}
where $\Symp^{K,\mathit{fix}}(M) \subset \Symp^K(M)$ is the subgroup of those $\phi$ such that $\phi|F$ is the identity outside a compact subset, or equivalently $\partial_\infty\phi|K = \mathit{id}$.

\begin{remark} \label{th:browder}
If $\Diff^c(\bR^n)$ is the group of compactly supported diffeomorphisms, then the inclusion $\Diff^c(\bR^n) \subset \Diff^+(S^n)$ induces an isomorphism on $\pi_0$. Hence, one can define a homomorphism
\begin{equation}
\pi_0(\Diff^+(S^n)) \longrightarrow \pi_0(\Symp^{K,\mathit{fix}}(M))
\end{equation}
by realizing classes as diffeomorphisms of $\bR^n$ which are the identity outside a ball, embedding that ball into $L \setminus \{q\}$, and then taking the induced symplectic automorphism of the cotangent bundle.
\end{remark}

\section{Topological and symplectic invariants\label{sec:invariants}}

Write $\bar{M} = M \cup \partial_\infty M$. Consider the short exact sequence
\begin{equation} \label{eq:relative-homology}
\xymatrix{
0 \ar[r] & H_n(L) \ar[r] & H_n(\bar{M},K) \ar[r] & H_{n-1}(K) \ar[r] & 0 \\
0 \ar[r] & \bZ \ar[u]^-{\iso} \ar[r] & \bZ^2 \ar[u]^-{\iso} \ar[r] & \bZ \ar[u]^-{\iso} \ar[r] & 0 
}
\end{equation}
The generators for the middle group in \eqref{eq:relative-homology} are the classes of the zero-section $L$ and the fibre $\bar{F} = F \cup K$. Moreover, $H_n(\bar{M},K)$ carries a non-symmetric intersection pairing, defined as follows: given two relative cycles, one perturbs the first one so that its boundary moves in direction of a Reeb vector field on $\partial_\infty M$, and then intersects that perturbed version with the (unperturbed) second cycle. Let's choose the orientations of $L$ and $F$ in such a way that $[L] \cdot [F] = (-1)^{n(n+1)/2}$. Then the intersection pairing is given by
\begin{equation} \label{eq:explicit-pairing}
(-1)^{n(n+1)/2} \begin{pmatrix} \chi(L) & 1 \\ (-1)^n & 1 \end{pmatrix}.
\end{equation}

Any element of $\Symp^K(M)$ induces an automorphism of $H_n(\bar{M},K)$, which acts trivially on the quotient $H_{n-1}(K)$, and is compatible with \eqref{eq:explicit-pairing}. Hence, the induced map is necessarily of the form
\begin{equation} \label{eq:induced-map}
\begin{pmatrix} 1 & \ast \\ 0 & 1 \end{pmatrix}
\end{equation}
if $n$ is odd, respectively of the form
\begin{equation} \label{eq:induced-map-2}
\begin{pmatrix} 
-1 &  -2/\chi(L) \\ 0 & 1
\end{pmatrix} \quad \text{or} \quad
\begin{pmatrix}
1 & 0 \\ 0 & 1
\end{pmatrix}
\end{equation}
if $n$ is even (and where the first possibility in \eqref{eq:induced-map-2} can only happen if $2$ is divisible by the Euler characteristic of $L$). 

Let $\Diff^K(M)$ be the topological analogue of $\Symp^K(M)$, which we define to be the group of orientation-preserving diffeomorphisms of $\bar{M}$ which map $K \subset \partial \bar{M} = \partial_\infty M$ orientation-preservingly to itself. Obviously, if $\phi_*$ is not the identity, then $[\phi] \in \pi_0(\Diff^K(M))$ is nontrivial.

\begin{lemma} \label{th:pseudo-isotopy}
Suppose that $L$ is a homotopy sphere, with $n \geq 6$ even. Then, any $[\phi] \in \pi_0(\Diff^K(M))$ has finite order.
\end{lemma}

\begin{proof}[Sketch of proof]
Since $\phi$ is not symplectic, $\phi_*$ does not a priori have to be compatible with all of \eqref{eq:explicit-pairing}, but it still preserves part of it, namely the pairing between $H_n(\bar{M})$ and $H_n(\bar{M}, K)$. This is enough to ensure that after possibly passing to the square, $\phi_*$ is the identity. Let's assume that that is the case.

Again after passing to an iterate, we may assume that $\phi|K$ is the identity, by looking at the topological analogue of \eqref{eq:map-to-theta-n}. Because $\phi_* = \mathit{id}$, the intersection number between $\phi(F)$ and another fibre $F'$ is zero. After applying the Whitney trick and a compactly supported isotopy, we may therefore assume that $\phi(\bar{F}) \cap \bar{F}' = \emptyset$. At this point, one can think of $\phi|\bar{F}$ as a ``higher-dimensional long knot'', which means an embedding $B^n \rightarrow B^n \times B^n$ whose restriction to the boundary is the standard map $S^{n-1} \rightarrow \{0\} \times S^{n-1}$. Since the space of such embeddings is connected \cite[Proposition 3.9(2)]{budney}, we may apply another compactly supported isotopy and achieve that $\phi|F$ is in fact the identity. Next, the action of $\phi$ on the normal bundle to $F$ is described by a map $B^n \rightarrow \mathit{GL}_+(n,\bR)$, which we can homotop to a constant. By extending this homotopy, one can arrange that $\phi$ is the identity in a neighbourhood of $\bar{F}$.

At this point, we can think of $\phi(L)$ as another ``long knot'', where this time $B^n \times B^n$ appears as the complement of a neighbourhood of $\bar{F}$. One can therefore deform $\phi$ (keeping it the same near $\bar{F}$) until it becomes equal to the identity on $L$. At this point, the action of $\phi$ on the normal bundle to $L$ is described by a map $(B^n,S^{n-1}) \rightarrow (\mathit{GL}_+(n,\bR),\mathit{id})$. But under our assumptions, $\pi_n(SO(n))$ is finite, hence after passing to a further iterate, one can assume that this map is nullhomotopic. The final outcome is that $\phi$ can be made equal to the identity in a neighbourhood of $L \cup \bar{F}$.

The remaining possibly nontrivial part of $\phi$ can be thought of as a pseudo-isotopy on $B^n \times S^{n-1}$, which is a simply-connected compact $(2n-1)$-manifold with boundary. The pseudo-isotopy theorem \cite{cerf70} (or rather, its mild generalisation to manifolds with boundary) applies, providing a deformation of the pseudo-isotopy to the identity.
\end{proof}

We now turn to symplectic topology. For technical simplicity, let's assume that $H^1(L) = 0$. Then one can introduce the notions of graded symplectic automorphism and graded Lagrangian submanifold of $M$, in an essentially unique way \cite{seidel99}. There are analogous notions of graded Legendrian submanifold, and graded contact diffeomorphism, in $\partial_\infty F$ (in fact, both can be defined via the symplectization, where they reduce to their symplectic geometry counterparts). An automorphism $\phi \in \Symp^K(M)$ has a canonical grading $\phi^\sharp$, which is characterized by the property that (for any grading $F^\sharp$ of the fibre), 
\begin{equation} \label{eq:sharp}
F^\sharp = \phi^\sharp(F^\sharp) \quad \text{outside a compact subset.}
\end{equation}

\begin{lemma} \label{th:hf-2}
Let $(K_t)$ be a loop in $\scrK(\partial_\infty M)$ based at $K$. Suppose that this does not lift to a loop in the space of graded Legendrian spheres. Then its image under \eqref{eq:boundary-map} has infinite order in $\pi_0(\Symp^K(M))$.
\end{lemma}

\begin{proof}
Embed the Legendrian isotopy $(K_t)$ into one of contact diffeomorphisms of $\partial_\infty M$, denoted by $(\partial_\infty \phi_t)$, and then extend that to an isotopy $(\phi_t)$ in $\Symp(M)$. Then, the image of $(K_t)$ under \eqref{eq:boundary-map} is the map $\phi = \phi_1 \in \Symp^K(M)$.

The isotopy $(\phi_t)$ has a unique grading $(\phi_t^\sharp)$, such that $\phi_0^\sharp$ carries the trivial grading (is the identity in the graded symplectic automorphism group). We then have
\begin{equation}
\phi_1^\sharp(F^\sharp) = F^\sharp[k] \quad \text{outside a compact subset,} 
\end{equation}
where $[k]$ is the notation for $k$-fold downwards shift in the grading. The (even) integer $k$ is nonzero in our case, because otherwise we would get a graded lift of the loop $(K_t)$, contradicting the assumption. By comparing this with \eqref{eq:sharp}, one sees the preferred grading of $\phi \in \Symp^K(M)$ satisfies 
\begin{equation} \label{eq:2-gradings}
\phi^\sharp = \phi_1^\sharp[-k].
\end{equation}
One can choose $(\phi_t)$ to be the constant isotopy on an arbitrary large compact subset of $M$, in particular on a neighbourhood of $L$. Then $\phi_1^\sharp(L^\sharp) = L^\sharp$, which in combination with \eqref{eq:2-gradings} shows that $\phi^\sharp(L^\sharp) = L^\sharp[-k]$. Hence
\begin{equation}
\mathit{HF}^*(L^\sharp,\phi^\sharp(L^\sharp)) \iso \mathit{HF}^{*-k}(L^\sharp,L^\sharp).
\end{equation}
On the other hand, if $[\phi] \in \pi_0(\Symp^K(M))$ was trivial, $\phi^\sharp$ would be isotopic to the identity as a graded symplectic automorphism, which would imply that
\begin{equation}
\mathit{HF}^*(L^\sharp,\phi^\sharp(L^\sharp)) \iso \mathit{HF}^{*}(L^\sharp,L^\sharp).
\end{equation}
But $\mathit{HF}^*(L^\sharp,L^\sharp)$ is a nonzero finite-dimensional graded group (isomorphic to the ordinary cohomology of $L$), hence that is a contradiction. The same applies to the iterates of $\phi$.
\end{proof}

%\begin{proof}
%This is an immediate consequence of the isotopy invariance of Lagrangian Floer cohomology. In fact, one only needs invariance under (Hamiltonian) isotopies of the compact Lagrangian submanifold $L$: if $[\phi]$ is trivial, then 
%\begin{equation}
%\mathit{HF}^*(L^\sharp,\phi^\sharp(F^\sharp)) \iso \mathit{HF}^*((\phi^\sharp)^{-1}%(L^\sharp),F^\sharp) \iso \mathit{HF}^*(L^\sharp,F^\sharp).
%\end{equation}
%\end{proof}
%In the same way, one reaches the stronger conclusion that $\phi$ is not weakly isotopic to the identity in $\Symp^K(M)$.
%
%
%\begin{proof}
%NEED TO REVISE THAT
%Because of the isomorphism $\mathit{HF}^*(L,\phi(F)) \iso \mathit{HF}^*(\phi^{-1}(L),F)$ and the isotopy invariance of Lagrangian Floer cohomology (note that $L$ is compact and satisfies $H^1(L) = 0$), it follows that $\phi^{-1}(L)$ is not isotopic to $L$ as a graded Lagrangian submanifold. Hence, $\phi^{-1}$ is not isotopic to the identity in the group of graded symplectic automorphisms, which implies the corresponding fact in $\Symp_q(M)$.
%\end{proof}

\section{The Dehn twist\label{sec:heuristic}}

At this point, it is helpful to recall the situation for the cotangent bundle of the standard sphere, so let's temporarily set $L = S^n$ and $M = T^*S^n$. The Dehn twist (or Picard-Lefschetz transformation) $\tau$ is a symplectic automorphism of the cotangent bundle $M$. It is compactly supported, hence in our notation belongs to $\Symp^{K,\mathit{fix}}(M)$ for any $q$.

The square $\tau^2$ is particularly easy to describe geometrically. Equip $L$ with the round metric. Use that metric to identify $M \iso TL$, and consider the Hamiltonian function
\begin{equation} \label{eq:geodesic-flow}
H(x) = h(\|x\|^2),
\end{equation}
where $\|x\|$ is the length of a tangent vector, and the function $h$ satisfies $h(r^2) = r$ for $r \gg 0$. At infinity, the Hamiltonian flow $(\phi_t)$ of $H$ is the normalized geodesic flow, moving each tangent vector by parallel transport along its unit speed geodesic. We then set 
\begin{equation} \label{eq:tau2}
\tau^{-2} = \phi_{2\pi}.
\end{equation}
To define $\tau^{-1}$ itself, one composes $\phi_\pi$ with the involution of $M$ induced by the antipodal map on $L$. 

Let's look at the topological aspect first. In terms of \eqref{eq:induced-map} or \eqref{eq:induced-map-2}, the Picard-Lefschetz formula says that 
\begin{equation} \label{eq:picard-lefschetz}
\tau_* = \begin{pmatrix} 1 & -1 \\ 0 & 1 \end{pmatrix} \text{ for odd $n$,} \qquad
\tau_* = \begin{pmatrix} -1 & -1 \\ 0 & 1 \end{pmatrix} \text{ for even $n$.}
\end{equation}
If $n$ is odd, $\tau_*$ has infinite order. For $n = 2,6$, it is known that $\tau^2$ is isotopic to the identity by a compactly supported isotopy (of diffeomorphisms). This is not true for other even $n$ (in those cases, $\tau^2$ is not even homotopic to the identity by a compactly supported homotopy \cite[Theorem 1.21]{avdek12}), but after passing to some power $\tau^{2k}$, one can again find a compactly supported isotopy to the identity \cite[Theorem 3]{kauffman-krylov05} (in fact, $k = 4$ suffices; see the discussion in \cite[Introduction]{klein11}, which summarizes the relevant parts of \cite{kauffman-krylov05, stevens86}).

From a symplectic geometry viewpoint, the situation is as follows. The canonical grading $\tau^\sharp$ has the property that \cite[Lemma 5.7]{seidel99}
\begin{equation} \label{eq:self-shift}
\tau^\sharp(L^\sharp) = L^\sharp[1-n].
\end{equation}
As in Lemma \ref{th:hf-2}, this implies that $[\tau] \in \pi_0(\Symp^K(M))$ has infinite order for all $n>1$ (leaving $n = 1$ as an elementary exercise). 

It is instructive to look a little closer at how $\tau$ acts on Lagrangian submanifolds. By \cite[Appendix]{seidel98b},
\begin{equation} \label{eq:connect}
\tau(F) \htp L \# F.
\end{equation}
Here $\#$ is Lagrangian connected sum (also called Lagrangian surgery \cite{polterovich91}); and $\htp$ stands for compactly supported Lagrangian isotopy. Since $\tau(F)$ intersects $L$ transversally and in a single point, the argument can be iterated, yielding in particular
\begin{equation} \label{eq:connect-2}
\tau^2(F) \htp ((-1)^{n-1} L) \# L \# F.
\end{equation}
The sign $(-1)^{n-1}$ is intended to keep track of orientations (which is not strictly necessary, since forming $\#$ does not require orientations, but is helpful for our argument): if $n$ is odd, the two copies of the zero-section $S^n$ contribute to \eqref{eq:connect-2} with the same orientation, while for even $n$ the orientations are opposite. This of course agrees with \eqref{eq:picard-lefschetz}.

Now suppose that $L$ is an exotic $n$-sphere, and let's engage in some speculative thinking. Both $F$ and $L \# F$ are diffeomorphic to $\bR^n$, but such a diffeomorphism is necessarily nontrivial at infinity. Hence, if one has some $\phi \in \Symp^K(M)$ for which the analogue of \eqref{eq:connect} holds, then $\phi|F$ must be nontrivial at infinity. Indeed, with suitable choices of orientations, it follows that its image under \eqref{eq:map-to-theta-n} is precisely the class $[L] \in \Theta_n \iso \pi_0(\Diff^+(S^{n-1}))$. Similarly, if one looks at $\phi^2$ and supposes that the analogue of \eqref{eq:connect-2} holds, then the image of $\phi$ under \eqref{eq:map-to-theta-n} must be 
$[(-1)^{n-1} L \# L]$. If $n$ is even, this would  contradict the previous suggestion, unless $[L] \in \Theta_n$ has order $2$. The conclusion is then: {\em for even $n$, one should not expect to obtain an analogue of $\tau$ in $\Symp^K(M)$, unless $[L] \in \Theta_n$ has order $2$; on the other hand, if there is an analogue of $\tau^2$, one can expect this to lie in $\Symp^{K,\mathit{fix}}(M)$.} Following the same line of argument, one arrives at the following: {\em for odd $n$, there is no fundamental obstruction to having an analogue of $\tau$ itself; however, one does not expect the analogue of $\tau^2$ to lie in $\Symp^{K,\mathit{fix}}(M)$, unless $[L] \in \Theta_n$ has order $2$.} These are purely heuristic considerations (so the words {\em should}  and {\em expect} are probably too strong), but they are useful as intuitive guideposts.

\section{Construction via geodesic flow\label{sec:first-def}}

Let $L$ be an exotic $n$-sphere. We write it as $L = S_f$ as in \eqref{eq:glue-2-discs}. Take $F_\pm \subset M$ to be the cotangent fibres at the centers $q_\pm = 0 \in B_\pm$, and $K_\pm \subset \partial_\infty M$ the corresponding spheres at infinity. Equip $L$ with a Riemannian metric as in Lemma \ref{th:metric}. If we identify $\partial_\infty M$ with the unit sphere tangent bundle for that metric, the associated Reeb flow equals the geodesic flow. Denoting that flow by $(\partial_\infty \phi_t)$, we find that 
\begin{equation}
(\partial_\infty \phi_{\pi}) (K_\pm) = K_\mp.
\end{equation}
More precisely, if we use the coordinates $B_\pm \subset \bR^n$ to identify $K_\pm = S^{n-1}$, then 
\begin{equation} \label{eq:phi22}
\begin{aligned}
& \partial_\infty \phi_\pi | K_- = a  f, \\
& \partial_\infty \phi_\pi | K_+ = a  f^{-1},
\end{aligned}
\end{equation}
where $a \in O(n) \subset \Diff(S^{n-1})$ is the antipodal map. In particular, $\partial_\infty \phi_{2\pi}$ maps $K_\pm$ to itself in an orientation-preserving way. Set $q = q_-$ (hence $F = F_-$ and $K = K_-$).

\begin{definition}
Define $[\rho] \in \pi_0(\Symp^K(M))$ to be the image of the homotopy class of the loop $(\partial_\infty \phi_t(K))_{0 \leq t \leq 2\pi}$ in $\scrK(\partial_\infty M)$ under \eqref{eq:boundary-map}.
\end{definition}

Taking $H \in \smooth(M,\bR)$ as in \eqref{eq:geodesic-flow} and its Hamiltonian flow $(\phi_t)$, one finds that the associated family of contactomorphisms at infinity is the previously considered $(\partial_\infty\phi_t)$. Hence, an explicit representative can be defined as in \eqref{eq:tau2}:
\begin{equation} \label{eq:phi4}
\rho = \phi_{2\pi}.
\end{equation}
This makes it clear that $\rho$ generalizes the inverse square Dehn twist $\tau^{-2}$. Let's draw some immediate conclusions from the definition:
\begin{itemize} \itemsep1em

\item 
Since the representation $L = S_f$ is not intrinsic, the class $[\rho]$ is unique only up to conjugation with elements of $\pi_0(\Diff^+(L)) \rightarrow \pi_0(\Symp(M))$. It is unknown at present whether these conjugations act nontrivially (note that a similar issue already arises for the standard Dehn twist).

\item
By \eqref{eq:phi22}, the image of $[\rho]$ under \eqref{eq:map-to-theta-n} is 
\begin{equation} \label{eq:afaf}
[af^{-1}af] \in \pi_0(\Diff^+(S^{n-1})).
\end{equation}
If $n$ is even, $a$ is isotopic to the identity, hence so is $af^{-1}af$, which means that one can find a representative of $[\rho]$ lying in $\Symp^{K,\mathit{fix}}(M)$. If $n$ is odd, then $af^{-1}af$ is isotopic to $f^2$ by Lemma \ref{th:conjugate-with-reflection}. Hence, a representative of $[\rho]$ in $\Symp^{K,\mathit{fix}}(M)$ exists if and only if $[L]$ is of order $2$ in $\Theta_n$. This is in line with the expectations from Section \ref{sec:heuristic}.

\item
One can assume that the function $h$ appearing in \eqref{eq:geodesic-flow} has the property that $k(r) = \frac{d}{dr} (h(r^2)) = h'(r^2) 2r$ itself has positive derivative $k'(r)>0$ for all $r>0$ such that $k(r) < 1$. In that case, if $\tilde{F} = T^*_{\tilde{q}}L$ is another cotangent fibre close to $F$, the intersection $\phi_{2\pi}(F) \cap \tilde{F}$ consists of two (transverse) points, corresponding to the two geodesics of length $<2\pi$ going from $q$ to $\tilde{q}$. The local sign of each such intersection point is determined by the Morse index of the geodesic mod $2$: it is $(-1)^{n(n+1)/2}$ if the Morse index is even, and $(-1)^{n(n+1)/2-1}$ otherwise. In our case, the Morse indices are $0$ and $n-1$, hence
\begin{equation} \label{eq:sigma-star}
\rho_*(F) \cdot \tilde{F} = \begin{cases} (-1)^{n(n+1)/2}\,2 & \text{$n$ odd}, \\
0 & \text{$n$ even.} \end{cases}
\end{equation}
This, together with the fact that $\rho|L$ is the identity, implies that the induced map $\rho_*$ on $H_n(\bar{M},K)$ agrees with the inverse square of \eqref{eq:picard-lefschetz}:
\begin{equation} \label{eq:rho-action}
\rho_* = \begin{pmatrix} 1 & 2 \\ 0 & 1 \end{pmatrix} \text{ for odd $n$,} \qquad
\rho_* = \begin{pmatrix} 1 & 0 \\ 0 & 1 \end{pmatrix} \text{ for even $n$.}
\end{equation}
In particular, if $n$ is odd then $[\rho] \in \pi_0(\Diff^K(M))$ has infinite order, just like for a standard Dehn twist. On the other hand, for even $n$, it follows from Lemma \ref{th:pseudo-isotopy} that the order of $[\rho] \in \pi_0(\Diff^K(M))$ must be finite.

\item
If we equip $F_\pm$ with the orientations induced by an orientation of $L$ itself, then the maps in \eqref{eq:phi22} preserve orientations if $n$ is odd, and reverse them if $n$ is even. If $n$ is odd, we can take any $\psi \in \Diff^+(L)$ which maps $x_+$ to $x_-$ and is isotopic to the identity, and consider
\begin{equation} \label{eq:pseudo-dehn}
(T^*\psi) \phi_{\pi} \in \Symp^K(M),
\end{equation}
which is an analogue of $\tau^{-1}$. If $n$ is even and $[L] \in \Theta_n$ has order $2$, there are orientation-reversing diffeomorphisms of $L$. We can then define \eqref{eq:pseudo-dehn} as before, but with an orientation-reversing $\psi$ (which is no longer unique up to isotopy). It is not clear in either case whether the square of \eqref{eq:pseudo-dehn} would be isotopic to the previously defined $\rho$; for that reason, we will not pursue these constructions further.
\end{itemize}

We now add gradings to our discussion. As in the proof of Lemma \ref{th:hf-2}, there is a unique lift $(\phi_t^\sharp)$ of $(\phi_t)$ to the graded symplectic automorphism group, such that $\phi_0^\sharp$ is the identity. With respect to the natural action of graded symplectic automorphisms on graded Lagrangian submanifolds, we then have for any grading of $F^\sharp$:

\begin{lemma} \label{th:shift}
Outside a compact subset, $\phi_{2\pi}^\sharp(F^\sharp)$ agrees with $F^\sharp[2-2n]$.
\end{lemma}

\begin{proof}
It is convenient to choose a particular way of thinking about gradings, namely that which starts by specifying a Lagrangian subbundle ${\mathcal F} \subset TM$ \cite[Example 2.10]{seidel99}. In our case, this will be the tangent bundle along the fibres of the projection $M \rightarrow L$. To determine the amount of shift of $\phi_{2\pi}^\sharp(F^\sharp)$ with respect to $F^\sharp$, one proceeds as follows. Take a point $x = (p,q) \in F$ sufficiently close to infinity, and its orbit $c(t) = \phi_t(x)$, $t \in [0,2\pi]$. We have Lagrangian subbundles $\Lambda_0,\Lambda_1 \subset c^*TM$, namely
\begin{equation}
\begin{aligned}
& \Lambda_{0,t} = D\phi_t(TF_x), \\
& \Lambda_{1,t} = {\mathcal F}_{c(t)}.
\end{aligned}
\end{equation}
These agree at the endpoints $t = 0,2\pi$. Take $I(\Lambda_0,\Lambda_1) \in \bZ$ to be the Maslov index for paths from \cite[Section 3]{robbin-salamon93}. Then, essentially by definition of $(\phi_t^\sharp)$, one has 
\begin{equation}
\phi_{2\pi}^\sharp(F^\sharp) = F^\sharp[-I(\Lambda_0,\Lambda_1)]
\end{equation}
locally near $x$. Everywhere along our path, the radial tangent direction $Z_{c(t)}$ is contained in $\Lambda_{0,t} \cap \Lambda_{1,t}$. We can consider the reduced spaces $\overline{TM} = \{X \in TM \;:\; \omega_M(X,Z) = 0\}/\bR Z$ at the points $c(t)$, and the corresponding Lagrangian subspaces $\bar\Lambda_{k,t} = \Lambda_{k,t}/\bR Z$. Then $I(\Lambda_0,\Lambda_1) = I(\bar\Lambda_0,\bar\Lambda_1)$. Next, by definition, $I(\bar\Lambda_0,\bar\Lambda_1)$ is a sum of contributions from points $t$ where $\bar\Lambda_{0,t} \cap \bar\Lambda_{1,t} \neq 0$, with endpoints weighted by $1/2$. These are precisely the conjugate points along the geodesic starting at $q$ and going in direction $p$ (with unit speed), and they contribute with their multiplicity \cite{duistermaat76}. In our case, because of the properties of the geodesic flow, $t = 0, \pi, 2\pi$ are the relevant conjugate points, with multiplicity $n-1$ each, giving $I(\bar\Lambda_0,\bar\Lambda_1) = 2n-2$.
\end{proof}

Lemma \ref{th:shift} says that if we lift $(\partial_\infty \phi_t(K))$ to a path of graded Legendrian submanifolds, then the two endpoints will have gradings which differ by $2n-2$. Lemma \ref{th:hf-2} then shows that $[\rho] \in \pi_0(\Symp^K(M))$ has infinite order, which proves Theorem \ref{th:main} for this construction.

\section{Weinstein handles\label{sec:weinstein}}

Let $M^{2n}$ be a symplectic manifold which is Liouville. By this we mean that $M$ carries a Liouville (symplectically expanding) vector field $Z$, whose flow exists for all times, and an exhausting function $h$ such that $Z.h>0$ outside a compact subset. In the special case where $h$ is Morse and $Z$ is gradient-like on all of $M$, we say that $M$ is Weinstein. The Weinstein structure gives rise to a description of $M$ by iterated attachment of Weinstein handles \cite{weinstein91} (a recent comprehensive reference is \cite{cieliebak-eliashberg}).

We will be interested in a very special class of Weinstein manifolds, namely those for which $h$ has only two critical points $x_0,x_1$, of index $0$ and $n$, respectively. In that case, the resulting description of $M$ has the following particularly simple form: start with the closed ball $B^{2n} \subset \bR^{2n}$, with its standard symplectic structure; attach a Weinstein handle along a Legendrian sphere $K \subset S^{2n-1} = \partial B^{2n}$; and then complete the result by adding an infinite cone along the boundary, which recovers $M$ up to symplectic isomorphism. More precisely, what one needs for the handle attachment process is a Legendrian sphere $K$ together with a {\em parametrization}, which means an element
\begin{equation} \label{eq:para-k}
f \in \mathit{Diff}(K,S^{n-1})/O(n).
\end{equation}
%The core of the handle is a standard ball $B^n$, whose boundary will be attached to $K$ via \eqref{eq:para-k}. 
%On the other hand, because we are in the ``critical dimension'' $n$, no additional bundle data is needed. 
One can describe the handle data more explicitly in terms of the flow of $Z$. Namely, let $\scrM(x_0,x_1)$ be the space of (unparametrized) flow lines of $Z$ with asymptotics $x_0,x_1$. For technical simplicity, suppose that there are Darboux coordinates around $x_0$ in which $Z = \myhalf p \partial_p + \myhalf q \partial_q$, and similarly Darboux coordinates around $x_1$ in which $Z = - \myhalf p \partial_p + \frac{3}{2} q \partial_q$. Then, by intersecting flow lines with small spheres in those coordinates, one obtains two maps
\begin{equation} \label{eq:asymptotic-maps}
\epsilon_0,\, \epsilon_1: \scrM(x_0,x_1) \longrightarrow S^{2n-1}.
\end{equation}
Each is an embedding, whose image is a Legendrian sphere. However, while the image of $\epsilon_1$ is always the same, namely $S^{n-1} = \{q = 0\} \subset S^{2n-1}$, that of $\epsilon_0$ is the sphere $K$ used in the handle attachment process. The parametrization \eqref{eq:para-k} is given (up to isotopy within that group) by 
\begin{equation} \label{eq:f}
f = \epsilon_1 \epsilon_0^{-1}. 
\end{equation}
% \epsilon_0^{-1}: S^{n-1} -> \partial D_0, \epsilon_1^{-1}: S^{n-1} -> \partial D_1
% the gluing map is 

\begin{lemma} \label{th:weinstein}
Let $L$ be a homotopy $n$-sphere. Then $T^*L$ is an instance of the previously described handle attachment process, where $K = S^{n-1} \subset S^{2n-1}$ is the standard Legendrian sphere, and the parametrization \eqref{eq:para-k} is the preimage of $L$ under \eqref{eq:diff-to-theta}.
\qed
\end{lemma}

The proof is elementary, and we will only describe the strategy. Take a Morse function on $L$ which has only two critical points, a minimum $x_0$ and a maximum $x_1$. Starting from that, one constructs a Morse function on $T^*L$, and a Liouville vector field which is gradient-like for that function. This is done so that all flow lines on $\scrM(x_0,x_1)$ lie inside the zero-section $L$, and will be gradient flow lines of the original Morse function. Then, the construction of the parametrization from \eqref{eq:asymptotic-maps} recovers the description of $L$ as in \eqref{eq:glue-2-discs}.

\begin{remark} \label{th:cotangent-fibre}
Weinstein manifolds $M$ in our class carry a specific non-compact Lagrangian submanifold, namely the unstable manifold of $x_1$ under the flow of $Z$. From the argument for Lemma \ref{th:weinstein} it follows that in the case $M = T^*L$, this can be identified with a cotangent fibre.
\end{remark}

%\begin{remark}
%Just as a consistency check, let's see that the correspondence described in Lemma \ref{th:weinstein} is well-defined. Given $L$ as an {\em unoriented} homotopy $n$-sphere, we can recover $f$ up to isotopy and composition with an orientation-reversing $a \in \mathit{O}(n)$ on either side. Composition on the left becomes irrelevant in the quotient \eqref{eq:para-k}. Composition on the right is also irrelevant for the handle attachment construction in this particular case, since the action of $a$ on $K = S^{n-1} \subset S^{2n-1}$ can be extended to a symplectic automorphism of the ball $B^{2n}$. 
%\end{remark}

\section{Lefschetz fibrations}

Let $\pi: E \rightarrow D$ be an exact symplectic Lefschetz fibration over the closed unit disc $D \subset \bC$ (the definition of Lefschetz fibration used here is as in \cite[Section 15]{seidel04}, except that we require the smooth fibres to be Liouville domains). We pick the base point $\ast = 1 \in \partial D$, and denote the fibre over that point by $P$. As mentioned before, this is a Liouville domain, while the total space $E$ itself is an exact symplectic manifold with corners.

Recall that a vanishing path is an embedded path in $D$ starting from one of the critical values, ending at $\ast$, and which otherwise avoids all critical values. To each such path belongs a Lefschetz thimble, which is an embedded Lagrangian disc $\Delta \subset E$ fibered over that path, and its vanishing cycle, which is the Lagrangian sphere $V = \partial \Delta \subset P$. Because of the way in which it arises as a boundary, the vanishing cycle comes with a parametrization in the sense of \eqref{eq:para-k}. To be precise, the parametrization is unique up to deformation in $\mathit{Diff}(V,S^{n-1})/O(n)$. 

We will often work with a fixed basis of vanishing paths $(\gamma_0,\dots,\gamma_r)$ (called a {\em distinguished basis} in the Picard-Lefschetz theory literature, see for instance \cite[Section 16d]{seidel04}). This determines associated bases of Lefschetz thimbles $(\Delta_0,\dots,\Delta_r)$, and of vanishing cycles $(V_0,\dots,V_r)$ with their parametrizations $(f_0,\dots,f_r)$. The classes of $(\Delta_0,\dots,\Delta_r)$, for some choice of orientations, form a basis of
\begin{equation} \label{eq:h-basis-by-thimbles}
H_n(E,P) \iso \bZ^{r+1}. 
\end{equation}
From $P$ and a basis of Lefschetz thimbles, one can reconstruct the symplectic Lefschetz fibration up to a suitable notion of deformation \cite[Lemma 16.9]{seidel04}.

One can turn the total space into a Liouville manifold, as follows. The given exact symplectic form $\omega_E = d\theta_E$ comes with a Liouville vector field, which points inwards along the fibrewise boundary (called the horizontal boundary $\partial_hE$ in \cite{seidel04}). By adding a multiple of the pullback of an exact symplectic form on the base $D$, one can achieve that the Liouville vector field for the modified symplectic form also points inwards along the vertical boundary $\partial_vE = \pi^{-1}(\partial D)$. At this point, rounding the corners yields a Liouville domain, to whose boundary one can then attach a semi-infinite cone. Denote the resulting Liouville manifold by $M$. This is specified uniquely (up to isomorphism of Liouville manifolds) by the same data as before ($P$ and a basis of Lefschetz thimbles).

\begin{lemma} \label{th:2-handle}
Take a Lefschetz fibration whose fibre $P$ is the ball cotangent bundle $B^*S^{n-1}$, and which has two vanishing cycles $(V_0,V_1)$, which are both copies of the zero-section $S^{n-1}$ but with arbitrary parametrizations $([f_0],[f_1])$. Then, the associated Liouville manifold $M$ can also be obtained by taking the standard ball $B^{2n}$ and attaching a handle along the standard Legendrian sphere $S^{n-1} \subset S^{2n-1}$, with parametrization $f_1f_0^{-1}$.
\end{lemma}

This is the result of applying a general translation process \cite[Section 8]{bourgeois-ekholm-eliashberg09}, which interprets the way in which $M$ is built from vanishing cycle data as a special case of Weinstein handle attachment. We omit the details, and only give an outline of the argument. As a preliminary observation, note that passing from $(f_0,f_1)$ to $(f_0g,f_1g)$ for any $g \in \Diff(S^{n-1})$ does not change $M$; it merely changes the way in which one thinks of the fibre $P$ as being identified with $B^*S^{n-1}$. Hence, we may assume without loss of generality that $(f_0,f_1) = (\mathit{id},f)$. Applying the strategy from \cite{bourgeois-ekholm-eliashberg09} directly yields a slightly more complicated Weinstein handle decomposition, which is as follows. Start with $D \times B^*S^{n-1}$, and round off its corners to obtain a Liouville domain $U$. Then, one obtains $M$ by attaching handles along the Legendrian spheres $K_0 = \{z_0\} \times S^{n-1}$, $K_1 = \{z_1\} \times S^{n-1}$ in $\partial U$ (and then adding an infinite cone to the boundary). Here, $z_0 \neq z_1$ are points on $S^1 = \partial D$; both $S^{n-1}$ are the zero-section; and one additionally equips $K_0,K_1$ with the parametrizations $(f_0,f_1) = (\mathit{id},f)$. A simple handle cancellation argument (see \cite[Proposition 12.22]{cieliebak-eliashberg} for the general notion) shows that attaching a handle along $K_0$ results in a Liouville domain which is deformation equivalent to $B^{2n}$. The deformation happens in such a way that $K_1$ becomes the standard Legendrian sphere $S^{n-1} \subset \partial B^{2n}$, still carrying the parametrization $f$.

\begin{lemma} \label{th:tstarl}
In the situation of Lemma \ref{th:2-handle}, $M$ is symplectically isomorphic to the cotangent bundle $T^*L$ of the homotopy $n$-sphere associated to $f_1f_0^{-1} \in \mathit{Diff}(S^{n-1})$, in the sense of \eqref{eq:glue-2-discs} (note that since $f_0$, $f_1$ are a priori only determined up to left composition with $O(n)$, one can always choose representatives such that $f_1f_0^{-1}$ is orientation-preserving).
\end{lemma}

This is immediate by combining Lemmas \ref{th:weinstein} and \ref{th:2-handle}. Alternatively, one can view it as an instance of a more general relation between real and complex Morse theory, explored in \cite{johns08, johns09b}. However, while it is highly plausible that the total spaces of the complexified Morse functions constructed in those papers are cotangent bundles, that fact has not actually been proved \cite[Remark 1.3]{johns09b}. This is the reason why we have adopted an approach which is less direct, but uses only standard tools.

From an exact Lagrangian submanifold $V \subset P \setminus \partial P$, one can construct a Legendrian submanifold of $\partial_\infty M$. This is obvious in the case where
\begin{equation} \label{eq:theta-vanishes}
\theta_P|V = 0,
\end{equation}
since then one can just round off the corners of $E$ as before, and $V$ becomes Legendrian in the resulting Liouville domain. In the general case, one has to first deform the one-form on $E$ by an exact amount to achieve \eqref{eq:theta-vanishes}, and then compensate that by another deformation at the end of the process (using Gray's theorem). As suggested by the terminology, the main application is to vanishing cycles $V = \partial \Delta$. In that case, we actually get a Lagrangian submanifold $\bR^n \iso F \subset M$ which, outside a compact subset, is modelled on our Legendrian.

\begin{lemma} \label{th:tstarl2}
In the situation of Lemma \ref{th:tstarl}, consider the vanishing cycle $V = V_1$ and its Lefschetz thimble $\Delta = \Delta_1$. Then, the associated Lagrangian submanifold $F \subset M$ can be identified with a cotangent fibre under $M \iso T^*L$.
\end{lemma}

This follows from Remark \ref{th:cotangent-fibre} and inspection of the argument for Lemma \ref{th:2-handle}.

\section{Monodromy}

Let's return to general exact Lefschetz fibrations. Take a vector field on the base $D$ which equals the (anticlockwise) rotational vector field along $\partial D$, and which vanishes in a neighbourhood of the critical values. One can lift the flow $(b_t)$ of that vector field via parallel transport to a family of fibrewise symplectic automorphisms $(\beta_t)$ of $E$. By construction, $\beta_{2\pi}$ maps $P$ to itself. We note two properties:
\begin{itemize} \itemsep1em
\item
$\beta_{2\pi}|P$ is the global monodromy of the Lefschetz fibration. In terms of a basis of vanishing cycles, the Picard-Lefschetz theorem then says that
\begin{equation} \label{eq:total-monodromy}
\beta_{2\pi}|P \htp \tau_{V_0} \cdots \tau_{V_r},
\end{equation}
where $\htp$ is an isotopy in $\Symp(P,\partial P)$ (in fact it is Hamiltonian rel $\partial P$, meaning that it is induced by a time-dependent Hamiltonian function which vanishes on  $\partial P$). Also, this isotopy is itself essentially canonical, meaning unique up to deformation rel endpoints. 

\item
With respect to \eqref{eq:h-basis-by-thimbles}, the induced map on $H_*(E,P)$ is given by
\begin{equation} \label{eq:monodromy-formula}
(\beta_{2\pi})_* = (-1)^n (A^t)^{-1}A. 
\end{equation}
Here, $A$ is a lower-triangular matrix with entries determined by the intersection numbers of the vanishing cycles, as follows:
\begin{equation}
A_{ij} = \begin{cases} (-1)^{(n+1)n/2} V_i \cdot V_j & i > j, \\
1 & i = j, \\
0 & i < j.
\end{cases}
\end{equation}
This is a form of the classical monodromy formula (\cite[vol. 2, Theorem 2.6]{arnold-gusein-zade-varchenko} or \cite[Hauptsatz]{lamotke75}).
\end{itemize}

\begin{assumption} \label{th:return-map}
Fix a vanishing cycle $V$, with the following property: there is an isotopy of oriented exact Lagrangian spheres in $P$,
\begin{equation} \label{eq:return-map}
\tau_{V_0} \cdots \tau_{V_r}(V) \htp V.
\end{equation}
\end{assumption}

From such an isotopy, one constructs a symplectic isomorphism $\phi \in \Symp^K(M)$, where $M$ is the Liouville manifold obtained from the total space $E$, and $K \subset \partial_\infty M$ the Legendrian submanifold associated to $V$. The construction is simple, and can be thought of as a parametrized version of that of $K$: as a consequence of \eqref{eq:total-monodromy},
\begin{equation} \label{eq:total-monodromy-2}
\beta_{2\pi}(V) \htp \tau_{V_0} \cdots \tau_{V_r}(V).
\end{equation}
By combining $(\beta_t(V))_{0 \leq t \leq 2\pi}$ with \eqref{eq:total-monodromy-2} and \eqref{eq:return-map}, one gets a loop of oriented exact Lagrangian submanifolds, each lying in a boundary fibre of the Lefschetz fibration. One turns these into a loop of oriented Legendrian submanifolds of $\partial_\infty M$ (starting and ending at $K$); then embeds that loop into a contact isotopy; and finally extends that to a symplectic isotopy of $M$. The map obtained at the endpoint of that isotopy is the desired $\phi$. We note the following:
\begin{itemize} \itemsep1em
\item
While $\phi$ depends on auxiliary choices made during the construction, the class $[\phi] \in \pi_0(\Symp^K(M))$ is well-defined. This means that it depends only on the given Lefschetz fibration, the choice of $V$, and the isotopy \eqref{eq:return-map}. 

\item If \eqref{eq:return-map} can be lifted to an isotopy of Lagrangian embeddings $S^n \rightarrow P$, one can achieve that $\phi \in \Symp^{K,\mathit{fix}}(M)$.

\item There is a commutative diagram
\begin{equation} \label{eq:include-diagram}
\xymatrix{
H_*(\bar{M}, K) \ar[d]_-{\phi_*}
\ar^-{\iso}[r] & 
H_*(E,V) 
\ar[r] & 
H_*(E,P) \ar[d]^-{(\beta_{2\pi})_*}
\\
H_*(\bar{M}, K) 
\ar^-{\iso}[r] & 
H_*(E,V) 
\ar[r] & 
H_*(E,P).
}
\end{equation}
\end{itemize}

\begin{remark} \label{th:eating-crow}
At this point, it may be worth while discussing under what circumstances one can get symplectic automorphisms of $M$ which are actually compactly supported. Suppose that \eqref{eq:return-map} can be strengthened to the following: there are isotopies (of exact Lagrangian spheres, compatible with parametrizations in the sense of \eqref{eq:para-k})
\begin{equation} \label{eq:vi}
\tau_{V_0} \cdots \tau_{V_r}(V_i) \htp V_i \quad
\text{for $i = 0,\dots,r$.}
\end{equation}
In that case, one can construct a diffeomorphism of $E$ which maps each fibre to itself, is fibrewise symplectic, and whose restriction to $P$ is the inverse of \eqref{eq:total-monodromy} (to be precise, this may only be possible if the symplectic form near the critical points is chosen to be standard in holomorphic Morse charts; we assume from now on that this technical condition is satisfied). Composing this with $\beta_{2\pi}$ yields another lift of $b_{2\pi}$ to a fibrewise symplectic automorphism, whose restriction to $P$ is the identity. Let's denote this lift by $\tilde{\beta}_{2\pi}$.

By looking at how $\tilde{\beta}_{2\pi}$ behaves on the boundary fibres, one gets a class 
\begin{equation} \label{eq:loop}
[t \mapsto 
(\beta_t|P)^{-1} \circ (\tilde{\beta}_{2\pi}|\pi^{-1}(e^{it})) \circ (\beta_t|P)] \in 
\pi_1(\Symp(P, \partial P)). 
\end{equation}
Let's also suppose for simplicity that $H^1(P,\partial P) = 0$, so that there is no difference between symplectic and Hamiltonian isotopy rel $\partial P$. If the class \eqref{eq:loop} is trivial, one can deform $\tilde{\beta}_{2\pi}$ to be the identity on $\pi^{-1}(\partial D)$. From that, one can then produce a compactly supported symplectic automorphism of the Liouville manifold $M$.

In terms of the original data entering the construction, the class \eqref{eq:loop} can be described as follows. For each $i$, the choice of an isotopy \eqref{eq:vi} gives rise to an isotopy inside $\Symp(P, \partial P)$,
\begin{equation}
\tau_{V_i}^{-1} (\tau_{V_0} \cdots \tau_{V_r}) \tau_{V_i} \htp
\tau_{V_0} \cdots \tau_{V_r}. 
\end{equation}
The composition of all those isotopies yields an isotopy between $\tau_{V_0} \cdots \tau_{V_r}$ and itself, which is related to \eqref{eq:loop} by composition with $\tau_{V_0} \cdots \tau_{V_r}$ (in the first version of this paper, it was erroneously assumed that \eqref{eq:loop} would automatically be contractible).
\end{remark}

We now add gradings to our discussion.

\begin{assumption} \label{as:grading-argument}
Assume that $c_1(E) = 0$ and $H^1(E) = 0$, so that we can introduce gradings on $E$ in an essentially unique way (these then induce gradings on $P$). Suppose that we have a vanishing cycle $V = \partial \Delta$ satisfying \eqref{eq:return-map}, as well as a closed Lagrangian submanifold $L \subset E$ with $H^1(L) = 0$.
\end{assumption}

\begin{lemma} \label{th:grading-argument}
In the situation of Assumption \ref{as:grading-argument}, suppose that the graded lift of \eqref{eq:return-map} to an isotopy
\begin{equation} \label{eq:return-map-2}
\tau_{V_0}^\sharp \cdots \tau_{V_r}^\sharp(V^\sharp) \htp V^\sharp[k]
\end{equation}
has $k \neq 2$. Then $[\phi] \in \pi_0(\Symp^K(M))$ has infinite order.
\end{lemma}

\begin{proof}
Suppose for simplicity that $(b_t)$ is rotation near $\partial D$. From $\beta_t(\Delta)$ and \eqref{eq:return-map-2}, we get a path of Lagrangian submanifolds in the total space $E$, whose endpoints (at times $t = 0$ and $t = 2\pi$) agree near the boundary. If we lift that to a path of graded Lagrangian submanifolds, then the gradings at the endpoints differ by $k -2$; the $k$ comes from \eqref{eq:return-map-2}, and the additional $2$ comes from the fact that near the boundary, the paths $b_t(\pi(\Delta))$ undergo a full rotation.

The same holds for gradings of the loop of Legendrian submanifolds associated to our Lagrangian submanifolds. Hence, assuming $k \neq 2$, the same argument as in Lemma \ref{th:hf-2} applies (using the Lagrangian submanifold from Assumption \ref{as:grading-argument} instead of the zero-section) and yields the desired result.
\end{proof}

\section{Construction via Lefschetz fibrations\label{sec:second-def}}

Take the Lefschetz fibration $\pi: E \rightarrow D$ from Lemma \ref{th:2-handle}, where the fibre is $P = B^*S^{n-1}$, and the two vanishing cycles $V_0 = V_1 = S^{n-1}$ are parametrized by $f_0 = \mathit{id}$ and $f_1 = f$, respectively. By Lemmas \ref{th:tstarl} and \ref{th:tstarl2}, the associated Liouville manifold $M$ is the cotangent bundle $T^*L$ of the homotopy sphere $L = S_f$ associated to $f$. 

To determine the global monodromy of our Lefschetz fibration, we write down the Dehn twists along the vanishing cycles:
\begin{equation} \label{eq:twisted-twist}
\begin{aligned}
& \tau_{V_0} = \tau, \\ 
& \tau_{V_1} = (T^*f)^{-1}\, \tau\,(T^*f), 
\end{aligned}
\end{equation}
where as usual $T^*f$ is the symplectic automorphism induced by $f$. Hence, if we take $V = S^{n-1}$ to be another copy of the zero-section (seen as an oriented Lagrangian sphere in $P$), then
\begin{equation}
\tau_{V_0}\tau_{V_1}(V) = V,
\end{equation}
so Assumption \ref{th:return-map} is satisfied (with the constant isotopy). As explained there, this leads to the construction of a class in $\pi_0(\Symp^K(M))$, which we denote by $[\sigma]$ in this particular case. Here are some consequences of the definition:

\begin{itemize} \itemsep1em
\item The construction of $\sigma$ depends (roughly speaking) on writing $M = T^*L$ as the total space of a Lefschetz fibration, which is turn is the complexification of a Morse function on $L$. This means that $[\sigma] \in \pi_0(\Symp^K(M))$ is not necessarily unique.

\item From \eqref{eq:twisted-twist} one sees that $\tau_{V_0}\tau_{V_1}|V = af^{-1}af \in \mathit{Diff}^+(S^{n-1})$. This is the image of $[\sigma]$ under the map $\pi_0(\Symp^K(M)) \rightarrow \pi_0(\Diff^+(S^{n-1}))$ from \eqref{eq:map-to-theta-n}, and that has the same implications as in \eqref{eq:afaf}.

\item Since $[V]$ generates $H_n(P)$, the horizontal maps in \eqref{eq:include-diagram} are isomorphisms. Hence we can use \eqref{eq:monodromy-formula} to determine the action of $\sigma$ on $H_*(\bar{M},K)$. The intersections of the two vanishing cycles $V_0 \cdot V_1$ vanishes if $n$ is even, and is $\pm 2$ if $n$ is odd. In the latter case, the sign is arbitrary, since it depends on the choice of orientations. The outcome is
\begin{equation}
\sigma_* = \begin{pmatrix} 3 & \pm 2 \\ \mp 2 & -1 \end{pmatrix} \text{ for odd $n$,} \quad
\sigma_* = \begin{pmatrix} 1 & 0 \\ 0 & 1 \end{pmatrix} \text{ for even $n$.} 
\end{equation}
These matrices are conjugate to those in \eqref{eq:rho-action}. This shows that if $n$ is odd, $[\sigma] \in \pi_0(\Diff^K(M))$ has infinite order. In contrast, for even $n$, $[\sigma] \in \pi_0(\Diff^K(M))$ has finite order, by Lemma \ref{th:pseudo-isotopy}.
\end{itemize}

\begin{lemma}
With respect to the canonical gradings of Dehn twists, one has
\begin{equation}
\tau_{V_0}^\sharp \tau_{V_1}^\sharp(V^\sharp) = V^\sharp[4-2n].
\end{equation}
\end{lemma}

This follows immediately from \eqref{eq:self-shift}, here applied to spheres of dimension $n-1$. In view of Proposition \ref{th:grading-argument} (with $L$ the zero-section), we see that $[\sigma] \in \pi_0(\Symp^K(M))$ has infinite order, which yields our second proof of Theorem \ref{th:main}.

\begin{remark}
Suppose that $n$ is even. Take a compactly supported symplectic isotopy $(\psi_t)$ with $\psi_0 = \tau$, and such that $\psi_1 = \mathit{id}$ on a neighbourhood of the zero-section (this exists because $a \in O(n)$ can be deformed to the identity). By making that neighbourhood sufficiently large, one can achieve that $\psi_1$ commutes with (the compactly supported maps) $\tau$ and $(T^*f)^{\pm 1}\tau(T^*f)^{\mp 1}$. This yields a loop in the compactly supported symplectic automorphism group $\Symp^c(T^*S^{n-1})$, namely
\begin{equation}
%\begin{aligned}
%t \longmapsto (T^*f)^{-1} \tau^{-1} (T^*f) \tau^{-1} \psi_t (T^*f)^{-1} \psi_t (T^*f)
%\tau (T^*f)^{-1} \tau \psi_t^{-1} (T^*f) \psi_t^{-1}.
t \longmapsto \tau^{-1} (T^*f) \tau^{-1} (T^*f)^{-1} \psi_t^{-1} 
%\\
%& \qquad \qquad 
(T^*f) \psi_t^{-1} \tau (T^*f)^{-1} \tau (T^*f) \psi_t (T^*f)^{-1} \psi_t.
%\end{aligned}
\end{equation}
Its homotopy class is a concrete instance of the element of $\pi_1(\Symp^c(T^*S^{n-1})) \iso \pi_1(\Symp(P,\partial P))$ appearing in \eqref{eq:loop}, which is an obstruction for constructing a compactly supported version of $\sigma$. Note that this homotopy class corresponds to one particular choice of isotopies \eqref{eq:vi}, and is not an invariant of the problem itself.
\end{remark}

%\bibliographystyle{plain}
%\bibliography{../../../bib/all}

\end{document}